\documentclass[preprint,12pt]{elsarticle}
\usepackage{amsmath}
\usepackage{amsthm}									
\usepackage{graphicx}
\usepackage{hyperref}
\usepackage{amssymb}
\usepackage{amsfonts}
\usepackage{float}
\usepackage{kbordermatrix}								
\usepackage[table]{xcolor}								

\newtheorem{thm}{Theorem}[section]
\newtheorem{lem}[thm]{Lemma}
\newtheorem{cor}[thm]{Corollary}

\theoremstyle{definition}									
\newtheorem{mydef}[thm]{Definition}
\newtheorem{rem}[thm]{Remark}
\newtheorem{ex}[thm]{Example}

\newproof{pf}{Proof}

\numberwithin{equation}{section}							

\newcommand{\R}{\mathbb R}								
\newcommand{\N}{\mathbb N}								
\newcommand{\C}{\mathbb C}								
\newcommand{\Z}{\mathbb Z}								
\newcommand{\ii}{\textup{i}}								

\newcommand{\nll}{\mathop{\rm null}}							

\newcommand{\sr}[1]{\rho\left(#1\right)}						
\newcommand{\sig}[1]{\sigma\left(#1\right)}						

\newcommand{\tth}{\text{th}}								
\newcommand{\mr}[1]{M_{#1}(\R)}							
\newcommand{\mc}[1]{M_{#1}(\C)}							
\newcommand{\jordan}[2]{J_{#1}{\left( #2 \right)}}				

\newcommand{\inv}[1]{#1^{-1}}							
\newcommand{\imod}[1]{\left( \bmod{#1} \right)}					
\journal{Linear Algebra and its Applications}

\begin{document}
\begin{frontmatter}
\title{Matrix Roots of Eventually Positive Matrices} 

\author[addy2]{Judith J. McDonald}										
\ead{jmcdonald@math.wsu.edu}
\ead[url]{http://www.math.wsu.edu/math/faculty/jmcdonald/}

\author[addy1]{Pietro Paparella\corref{corpp}}
\ead{ppaparella@wm.edu}
\ead[url]{http://ppaparella.people.wm.edu/}

\author[addy2]{Michael J. Tsatsomeros}
\ead{tsat@math.wsu.edu}
\ead[url]{http://www.math.wsu.edu/faculty/tsat/}

\cortext[corpp]{Corresponding author.}

\address[addy2]{Department of Mathematics, Washington State University, Pullman, WA 99164-1113, U.S.A.}
\address[addy1]{Department of Mathematics, College of William \& Mary, Williamsburg, VA 23187-8795, U.S.A.}

\begin{abstract}
Eventually positive matrices are real matrices whose powers become and remain strictly positive. As such, eventually positive matrices are {\it a fortiori} matrix roots of positive matrices, which motivates us to study the matrix roots of primitive matrices. Using classical matrix function theory and Perron-Frobenius theory, we characterize, classify, and describe in terms of the real Jordan canonical form the $p$th-roots of eventually positive matrices.
\end{abstract}

\begin{keyword}
matrix function \sep eventually positive matrix \sep primitive matrix \sep matrix root \sep Perron-Frobenius theorem \sep Perron-Frobenius Property \sep stochastic matrix 

\MSC[2008] 15B48 \sep 15A21
\end{keyword}
\end{frontmatter}

\section{Introduction} \label{intro}

A matrix $A \in \mr{n}$ is {\it eventually positive} ({\it nonnegative}) if there exists a nonnegative integer $p$ such that $A^k$ is entrywise positive (nonnegative) for all $k \geq p$. If $p$ is the smallest such integer, then $p$ is called the {\it power index of $A$} and is denoted by $p(A)$.

Eventually nonnegative matrices have been the subject of study in several papers \cite{cnm2002, cnm2004,f1978, jt2004, n2006, nt2009, zm2003, zt1999} and it is well-known that the notions of eventual positivity and nonnegativity are associated with properties of the eigenspace corresponding to the spectral radius.

A matrix $A\in \mr{n}$ has the {\it Perron-Frobenius property} if its spectral radius is a positive eigenvalue corresponding to an entrywise nonnegative eigenvector. The {\it strong Perron-Frobenius property} further requires that the spectral radius is simple; that it dominates in modulus every other eigenvalue of $A$; and that it has an entrywise positive eigenvector.

Several challenges regarding the theory and applications of eventually nonnegative matrices remain unresolved. For example, eventual positivity of $A$ is equivalent to $A$ and $A^T$ having the strong Perron-Frobenius property, however, the Perron-Frobenius property for $A$ and $A^T$ is a necessary but not sufficient condition for eventual nonnegativity of $A$.

An eventually nonnegative (positive) matrix with power index $p=p(A)$ is, {\it a fortiori}, a $p$th-root of the nonnegative (positive) matrix $A^p$. As a consequence, in order to gain more insight into the powers of an eventually nonnegative (positive) matrix, it is only natural to examine the roots of matrices that possess the (strong) Perron-Frobenius property. We begin this pursuit herein by characterizing the roots of matrices that possess the strong Perron-Frobenius property.

We proceed as follows: in \hyperref[theory]{Section \ref*{theory}}, we recall results concerning matrix functions and, for the sake of completeness and clarity, we present facts needed to analyze a matrix function via the real Jordan canonical form; we also use the real Jordan canonical form to give alternate proofs for \cite[Theorems 2.3 and 2.4]{hl2011}. In \hyperref[main_results]{Section \ref*{main_results}}, we recall results from the Perron-Frobenius theory of nonnegative matrices and (eventually) positive matrices. We characterize the eventually positive roots of a general primitive matrix, and illustrate our main results via examples. We also present a related result concerning {\it eventually stochastic} matrices. 

\section{Matrix roots via the complex and real Jordan canonical form} \label{theory}

We review some basic notions and results from the theory of matrix functions (for further results, see \cite{h2008}, \cite[Chapter 9]{hj1994}, or \cite[Chapter 6]{lt1985}).

Let $\jordan{n}{\lambda} \in \mc{n}$ denote the $n \times n$ Jordan block with eigenvalue $\lambda$. For $A \in \mc{n}$, let $J = \inv{Z} A Z = \bigoplus_{i=1}^t \jordan{n_i}{\lambda_i} = \bigoplus_{i=1}^t J_{n_i}$, where $\sum n_i = n$, denote its Jordan canonical form. Denote by $\lambda_1,\dots,\lambda_s$ the {\it distinct} eigenvalues of $A$, and, for $i=1,\dots,s$, let $m_i$ denote the {\it index} of $\lambda_i$, i.e., the size of the largest Jordan block associated with $\lambda_i$. Denote by $\ii$ the imaginary unit, i.e., $\ii := \sqrt{-1}$.

\begin{mydef} 
Let $f : \C \rightarrow \C$ be a function and let $f^{(k)}$ denote the $k$th derivative of $f$. The function $f$ is said to be {\it defined on the spectrum of $A$} if the values
\begin{align*}
\begin{array}{c c c}
f^{(k)}(\lambda_i), & k=0,\dots,m_i-1, & i=1,\dots,s,
\end{array}
\end{align*}
called {\it the values of the function $f$ on the spectrum of $A$}, exist.  
\end{mydef}

\begin{mydef}[Matrix function via Jordan canonical form] 
If $f$ is defined on the spectrum of $A \in \mc{n}$, then
\begin{align*} 
f(A) := Z f(J) \inv{Z} = Z \left( \bigoplus_{i=1}^t f(J_{n_i}) \right) \inv{Z}, 
\end{align*}
where
\begin{align}
f(J_{n_i}) := 
\begin{bmatrix} 
f(\lambda_i) & f'(\lambda_i) & \dots   & \frac{f^{(n_i-1)}(\lambda_i)}{(n_i - 1)!} 	\\
	          & f(\lambda_i)  & \ddots & \vdots 						\\
	          &                       & \ddots & f'(\lambda_i)						\\
	          &  	  	   &             & f(\lambda_i)
\end{bmatrix}. \label{fjb}
\end{align}
\end{mydef}

\begin{mydef} 
For $z = r \exp{\left( \ii \theta \right)} \in \C$, where $r >0$, and an integer $p >1$, let 
\begin{align*} 
z^{1/p} := r^{1/p} \exp{( \ii \theta/p )}, 
\end{align*}
and, for $j \in \left\{ 0, 1, \dots, p-1 \right\}$, define
\begin{align*}  
f_j (z) = z^{1/p} \exp{(\ii2 \pi j/ p)} = r^{1/p} \exp{\left( \ii \left[ \theta + 2 \pi j \right] / p \right)}, 
\end{align*}
i.e., $f_j$ is the $(j+1)$st-branch of the $p$th-root function.
\end{mydef}

Note that 
\begin{align}
f_j^{(k)} (z) = \frac{1}{p^k} \prod_{i=0}^{k-1} (1 - ip) \left[ r^{(1-kp)/p} \exp{\left( \ii [2 \pi j + \theta(1 - kp)]/p \right)} \right], \label{pth_rt_der_fx}
\end{align}
where $k$ is a nonnegative integer and the product $\prod_{i=0}^{k-1} (1 - ip)$ is empty when $k=0$.

Next we present several technical lemmas on the branches of the $p$th-root function.

\begin{lem} \label{lem_conj} 
For $z \in \C$, $\Im{(z)} \neq 0$, $j$, $j' \in \{0,1,\dots,p-1\}$, and $f_j^{(k)}$ as in \eqref{pth_rt_der_fx}, we have $f_j^{(k)} (z) = \overline{f_{j'}^{(k)} (\bar{z})}$ if and only if $j + j' \equiv 0 \imod{p}$. 
\end{lem}

\begin{proof}[\sc Proof] Note that
\begin{align*}
f_j^{(k)} (z) &= \overline{f_{j'}^{(k)} (\bar{z})} \\
\Longleftrightarrow \exp{\left( \ii [2 \pi j + \theta(1 - kp)]/p \right)} &= \exp{\left( \ii [-2 \pi j' + \theta(1 - kp)]/p \right)}	\\
\Longleftrightarrow [2 \pi j + \theta(1 - kp)]/p &= [-2 \pi j' + \theta(1 - kp)]/p  + 2 \pi \ell 						\\
\Longleftrightarrow 2 \pi j + \theta(1 - kp) &= 2 \pi (p \ell - j') + \theta(1 - kp) 								\\
\Longleftrightarrow j + j' &= \ell p 														\\
\Longleftrightarrow j + j' &\equiv 0 \imod{p},
\end{align*}
where $\ell \in \Z$. Finally, we remark that $j + j' \equiv 0 \imod{p} \Longleftrightarrow j = j' = 0$ or $j = p - j'$.
\end{proof}

\begin{lem} \label{lem_conj2} 
Let $z=r \exp{\left( \ii \pi \right)}$, $r > 0$. For $j$, $j' \in \{0,1,\dots,p-1\}$ and $f_j^{(k)}$ as in \eqref{pth_rt_der_fx}, we have $f_j^{(k)} (z) = \overline{f_{j'}^{(k)} (z)}$ if and only if $j + j' \equiv -1 \imod{p}$. 
\end{lem}

\begin{proof}[\sc Proof] Note that
\begin{align*}
f_j^{(k)} (z) &= \overline{f_{j'}^{(k)} (z)} 											\\
\Longleftrightarrow \exp{\left( \ii [2 \pi j + \pi(1 - kp)]/p \right)} &= \exp{\left( \ii [-2 \pi j' - \pi(1 - kp)]/p \right)}	\\
\Longleftrightarrow [2 \pi j + \pi(1 - kp)]/p &= [-2 \pi j' - \pi(1 - kp)]/p + 2 \pi \ell 					\\
\Longleftrightarrow 2 \pi j + \pi(1 - kp) &= 2 \pi \ell p - 2 \pi j' - \pi(1 - kp)							\\
\Longleftrightarrow 2 \pi [(j +  j') - (kp - 1)] &= 2 \pi \ell p									\\
\Longleftrightarrow j + j' &\equiv (kp - 1) \imod{p}										\\
\Longleftrightarrow j + j' &\equiv -1 \imod{p},
\end{align*} 
where $\ell \in \Z$. Finally, we remark that $j + j' \equiv -1 \imod{p} \Longleftrightarrow$ or $j + j' = p - 1$.
\end{proof}

The following theorem classifies all $p$th-roots of a general nonsingular matrix \cite[Theorems 2.1 and 2.2]{s2003}. 

\begin{thm}[Classification of $p$th-roots of nonsingular matrices] \label{thm_class_rts} 
If $A \in \mc{n}$ is nonsingular, then $A$ has precisely $p^s$ $p$th-roots that are expressible as polynomials in $A$, given by 
\begin{align}
X_j = Z \left( \bigoplus_{i=1}^t f_{j_i} (J_{n_i}) \right) \inv{Z}, \label{prim_roots}			
\end{align}
where $j = \begin{pmatrix} j_1, \dots, j_t \end{pmatrix}$, $j_i \in \{0,1,\dots,p-1\}$, and $j_i = j_k$ whenever $\lambda_i = \lambda_k$. 

If $s < t$, then $A$ has additional $p$th-roots that form parameterized families 
\begin{align}
X_j (U) = Z U \left( \bigoplus_{i=1}^t f_{j_i} (J_{n_i}) \right) \inv{U} \inv{Z}, \label{nnprim_roots} 
\end{align}
where $U$ is an arbitrary nonsingular matrix that commutes with $J$ and, for each $j$, there exist $i$ and $k$, depending on $j$, such that $\lambda_i = \lambda_k$, while $j_i \neq j_k$.
\end{thm}

In the theory of matrix functions, the roots given by \eqref{prim_roots} are called the {\it primary roots} of $A$, and the roots given by \eqref{nnprim_roots}, which exist only if $A$ is {\it derogatory} (i.e., some eigenvalue appears in more than one Jordan block), are called the {\it nonprimary roots} \cite[Chapter 1]{h2008}.

The next result provides a necessary and sufficient condition for the existence of a root for a general matrix, which is clearly satisfied by any nonsingular matrix (see \cite{p2002}).

\begin{thm}[Existence of $p$th-root] \label{asc_seq_thm}
A matrix $A \in \mc{n}$ has a $p$th-root if and only if the ``ascent sequence" of integers $d_1,d_2,\dots$ defined by
\begin{align*} 
d_i = \dim{(\nll{(A^i)})} - \dim{(\nll{(A^{i-1})})}
\end{align*}
has the property that for every integer $\nu \geq 0$ no more than one element of the sequence lies strictly between $p\nu$ and $p(\nu + 1)$.
\end{thm}

Before we state results concerning the matrix roots of a real matrix, we state some well-known results concerning the real Jordan canonical form for real matrices (see \cite[Section 3.4]{hj1990}, \cite[Section 6.7]{lt1985}).

\begin{thm}[Real Jordan canonical form] \label{rjcf_thm} 
If $A \in \mr{n}$ has $r$ real eigenvalues (including multiplicities) and $c$ complex conjugate pairs of eigenvalues (including multiplicities), then there exists a real, invertible matrix $R \in \mr{n}$ such that 
\begin{align*}
\inv{R} A R = J_{\R} =
\begin{bmatrix} 
\bigoplus_{k=1}^{r} J_{n_k} (\lambda_k) & \\ & \bigoplus_{k = r + 1}^{r + c} C_{n_k} (\lambda_k)  
\end{bmatrix},
\end{align*}
where:
\begin{enumerate}
\item
\begin{align}
C_{k} (\lambda) := 
\begin{bmatrix} 
C(\lambda) & I_2                        			\\
            & C(\lambda) & \ddots 				\\
            &                    & \ddots & I_2 		\\ 
            &                    &            & C(\lambda) 
\end{bmatrix} \in \mr{2k}; \label{Ck_lambda} 	
\end{align}
\item
\begin{align} 
C(\lambda) := 
\begin{bmatrix} 
\Re{(\lambda)} & \Im{(\lambda)} \\ 
-\Im{(\lambda)} & \Re{(\lambda)}  
\end{bmatrix} \in \mr{2}; \label{C_lambda}	
\end{align}
\item $\lambda_1, \dots, \lambda_r$ are the real eigenvalues (including multiplicities) of $A$; and 
\item $\lambda_{r+1}, \bar{\lambda}_{r + 1}, \dots, \lambda_{r + c}, \bar{\lambda}_{r + c}$ are the complex eigenvalues (including multiplicities) of $A$.
\end{enumerate}
\end{thm}

\begin{lem} \label{rjcf_lem} 
Let $\lambda \in \C$ and suppose $C_{k} (\lambda)$ and $C(\lambda)$ are defined as in \eqref{Ck_lambda} and \eqref{C_lambda}, respectively. If $S_k :=\left( \bigoplus_{i=1}^k S \right) \in \mr{2k}$, where $S := \begin{bmatrix} -\ii & -\ii \\ 1 & -1 \end{bmatrix}$, then 
\begin{align} 
\inv{S_k} C_k (\lambda) S_k = D_k (\lambda) := \begin{bmatrix} D(\lambda) & I_2 \\ & D(\lambda) & \ddots \\  & & \ddots & I_2 \\ & & & D(\lambda)\end{bmatrix} \in \mc{2k}, \label{Dk_lambda}
\end{align}
where $D(\lambda) := \begin{bmatrix} \lambda & 0 \\ 0 & \bar{\lambda} \end{bmatrix}$. 
\end{lem}

\begin{proof}[\sc Proof] 
Proceed by induction on $k$, the number of $2\times2$-blocks; when $k=1$ one readily obtains
\begin{align*}
\inv{S} C(\lambda) S = \frac{1}{2} \begin{bmatrix} \ii & 1 \\ \ii & -1 \end{bmatrix} \begin{bmatrix} \Re{(\lambda)} & \Im{(\lambda)} \\ -\Im{(\lambda)} & \Re{(\lambda)}  \end{bmatrix} \begin{bmatrix} -\ii & -\ii \\ 1 & -1 \end{bmatrix} = D(\lambda).
\end{align*}
Now assume the assertion holds for all matrices of the form \eqref{Ck_lambda} of dimension $2(k-1)$. Note that the matrices in the product $\inv{S_k} C_k (\lambda) S_k$ can be partitioned as
\begin{align*} 
\begin{bmatrix} \inv{S_{k-1}} & Z \\ Z^T & \inv{S} \end{bmatrix}
\begin{bmatrix} C_{k-1} (\lambda) & Y \\ Z^T & C(\lambda) \end{bmatrix} 
\begin{bmatrix} S_{k-1} & Z \\ Z & S \end{bmatrix},
\end{align*}
where $Z \in \mr{2(k-1),2}$ is a rectangular zero matrix, $Y= \begin{bmatrix} Z_2 \\ Z_2 \\ \vdots \\ Z_2 \\ I_2 \end{bmatrix} \in \mr{2(k-1),2}$, and $Z_2$ is the $2\times2$ zero matrix. With the above partition in mind, and following the induction hypothesis, we obtain 
\begin{align*}
\inv{S} C_k (\lambda) S
&= \begin{bmatrix} D_{k-1} (\lambda) & \inv{S_{k-1}} Y S \\ Z^T & \inv{S} C(\lambda) S \end{bmatrix}
\end{align*}
 and note that $\inv{S_{k-1}} Y S = Y$ and $\inv{S} C(\lambda) S = D(\lambda)$.
\end{proof}

\begin{lem} \label{rjcf_lem_2} 
Let $\lambda \in \C$ and suppose $D_k (\lambda)$ and $D(\lambda)$ are defined as in \hyperref[rjcf_lem]{\rm Lemma \ref*{rjcf_lem}}. If $P_k$ is the permutation matrix given by 
\begin{align}
P_k = \begin{bmatrix} e_1 & e_3 & \dots & e_{2k-1} & e_2 & e_4 & \dots & e_{2k} \end{bmatrix} \in \mr{2k}, \label{Perm_P_k}
\end{align} 
where $e_i$ denotes the canonical basis vector in $\R^n$ of appropriate dimension, then 
\begin{align*}
P_k^T D_k (\lambda) P_k = 
\jordan{k}{\lambda} \oplus \jordan{k}{\bar{\lambda}}. 	
\end{align*}
\end{lem}

\begin{proof}[\sc Proof] 
Proceed by induction on $k$: the base-case when $k=1$ is trivial, so we assume the assertion holds for matrices of the form $\eqref{Dk_lambda}$ of dimension $2(k-1)$. If $z_n$ denotes the $n \times 1$ zero vector, then
\begin{align*}
D(\lambda) = 
\begin{bmatrix} 
\lambda & e_2^T & 0 					\\ 
z_{2(k-1)} & D_{2(k-1)} (\bar{\lambda}) & e_{2(k-1)}	\\
0 & z^T_{2(k-1)} & \bar{\lambda}
\end{bmatrix},
\end{align*}
and if $\hat{P}$ is the permutation matrix defined by 
\begin{align*} 
\hat{P} := 
\begin{bmatrix}
1 & z_{2(k-1)}^T & 0 			\\
z_{2(k-1)} & P_{2(k-1)} & z_{2(k-1)}	\\
0 & z_{2(k-1)}^T & 1
\end{bmatrix} \in \mr{2k},
\end{align*}
then, following the induction-hypothesis, 
\begin{align}
\hat{P}^T D(\lambda) \hat{P} &= 
\begin{bmatrix}
\lambda & z_{k-1}^T                              & e_1^T  & 0 				\\
z_{k-1}  & \jordan{k-1}{\bar{\lambda}} & Z_{k-1}                               & z_{k-1} 	\\
z_{k-1}  & Z_{k-1}                                   & \jordan{k-1}{\lambda}       & e_{k-1} 	\\
0            & z_{k-1}^T                               & z_{k-1}^T                          & \lambda 
\end{bmatrix}. \label{27}
\end{align}
A permutation-similarity by the matrix $\bar{P}$ defined by 
\begin{align*}
\bar{P} :=
\begin{bmatrix}
1  &              &             & 	\\
    &              & I_{k-1} & 	\\
    & I_{k-1} &              & 	\\
    &              &             & 1
\end{bmatrix} \in \mr{2k}
\end{align*}
brings the matrix in the right-hand-side of $\eqref{27}$ to the desired form. The proof is completed by noting that 
\begin{align*}
\hat{P} \bar{P}
&= 
\begin{bmatrix}
e_1 & e_2 & \dots & e_{2(k-1)} & e_3 & \dots & e_{2k-1} & e_{2k}
\end{bmatrix}
\begin{bmatrix}
1 &             &             & 	\\
   &             & I_{k-1} & 	\\
   & I_{k-1} &             & 	\\
   &             &              & 1
\end{bmatrix}
= P_k,
\end{align*}
since right-hand multiplication by $\bar{P}$ permutes columns 2 through $k$ with columns $k+1$ through $2k-1$.
\end{proof}

\begin{cor} \label{rjcf_cor} 
Let $\lambda \in \C$, $\lambda \neq 0$, and let $f$ be a function defined on the spectrum of $\jordan{k}{\lambda} \oplus \jordan{k}{\bar{\lambda}}$.  For $j$ a nonnegative integer, let $f^{(j)}_\lambda$ denote $f^{(j)}(\lambda)$. If $C_k (\lambda)$ and $C(\lambda)$ are defined as in \eqref{Ck_lambda} and \eqref{C_lambda}, respectively, then 
\begin{align*} 
f ( C_{k} (\lambda) ) = 
\begin{bmatrix} 
C(f_\lambda) & C(f'_\lambda) & \dots & C \left( \frac{f^{(k-1)}_\lambda}{(k-1)!} \right)	\\
& C(f_\lambda) & \ddots & \vdots 									\\
& & \ddots & C(f'_\lambda) 										\\
& & & C(f_\lambda)
\end{bmatrix} \in \mr{2k}
\end{align*}
if and only if $\overline{f_\lambda^{(j)}} = f_{\bar{\lambda}}^{(j)}$.
\end{cor}

\begin{proof}[\sc Proof] 
Following Lemmas \hyperref[rjcf_lem]{\ref*{rjcf_lem}} and \hyperref[rjcf_lem_2]{\ref*{rjcf_lem_2}}, 
\begin{align}
P_k^T \inv{S_k} C_k (\lambda) S_k P_k = \jordan{k}{\lambda} \oplus \jordan{k}{\bar{\lambda}}. \label{Cor_Eg} 
\end{align} 
Since $f(A) = f(\inv{X} A X)$ (\cite[Theorem 1.13(c)]{h2008}), $f(A \oplus B) = f(A) \oplus f(B)$ \cite[Theorem 1.13(g)]{h2008}, and $\overline{f_\lambda^{(j)}} = f_{\bar{\lambda}}^{(j)}$ for all $j$, applying $f$ to \eqref{Cor_Eg} yields   
\begin{align*}
P_k^T \inv{S_k} f( C_k (\lambda)) S_k P_k 
= f(\jordan{k}{\lambda}) \oplus f(\jordan{k}{\bar{\lambda}}) 		
= f(\jordan{k}{\lambda}) \oplus \overline{f(\jordan{k}{\lambda})}.	
\end{align*}
Hence, 
\begin{align*}
\inv{S_k} f( C_k (\lambda)) S_k 
&= P_k \left[ f(\jordan{k}{\lambda}) \oplus \overline{f(\jordan{k}{\lambda})} \right] P_k^T	\\
&= 
\begin{bmatrix} 
D(f_\lambda) & D(f'_\lambda) & \dots & D \left( \frac{f^{(k-1)}_\lambda}{(k-1)!} \right)  	\\
		 & D(f_\lambda) & \ddots & \vdots  							\\
		 &                       & \ddots & D(f'_\lambda) 						\\
 		 &  		     &            & D(f_\lambda)
\end{bmatrix}
\end{align*}
and
\begin{align*}
f( C_k (\lambda))
&= 
S_k
\begin{bmatrix} 
D(f_\lambda) & D(f'_\lambda) & \dots & D \left( \frac{f^{(k-1)}_\lambda}{(k-1)!} \right)  	\\
		 & D(f_\lambda) & \ddots & \vdots  							\\
		 &                       & \ddots & D(f'_\lambda) 						\\
 		 &  		     &            & D(f_\lambda)
\end{bmatrix} \inv{S_k}											\\
&= 
\begin{bmatrix} 
C(f_\lambda) & C(f'_\lambda) & \dots   & C \left( \frac{f^{(k-1)}_\lambda}{(k-1)!} \right)	\\
                      & C(f_\lambda)  & \ddots & \vdots 							\\
                      &                        & \ddots & C(f'_\lambda) 						\\
                      &                        &            & C(f_\lambda)
\end{bmatrix}. 
\end{align*}
The converse follows from noting that, for $\mu$, $\nu \in \C$, the product
\begin{align*}
S \begin{bmatrix} \mu & 0 \\ 0 & \nu \end{bmatrix} \inv{S} 
&=  \begin{bmatrix} -\ii & -\ii \\ 1 & -1 \end{bmatrix} \begin{bmatrix} \mu & 0 \\ 0 & \nu  \end{bmatrix} \frac{1}{2} \begin{bmatrix} \ii & 1 \\ \ii & -1 \end{bmatrix}					\\
&= \frac{1}{2} \begin{bmatrix} -\ii \mu & -\ii \nu \\ \mu & -\nu \end{bmatrix} \begin{bmatrix} \ii & 1 \\ \ii & -1 \end{bmatrix}	\\
&= \frac{1}{2} \begin{bmatrix} \mu + \nu & \ii(\nu - \mu) \\ i(\mu - \nu) & \mu + \nu \end{bmatrix}
\end{align*}
is real if and only if $\bar{\nu} = \mu$.

Moreover, from our analysis, it also follows that, in general,  
\begin{align*} 
f(C_{k} (\lambda)) = 
\begin{bmatrix} 
f(C_\lambda) & f'(C_\lambda) & \dots   & \frac{f^{(k-1)}(C_\lambda)}{(k-1)!}  	\\
		 & f(C_\lambda) & \ddots & \vdots 					\\
		 & 		     & \ddots & f'(C_\lambda)				\\
		 & 		     & 	         & f(C_\lambda)
\end{bmatrix} \in \mc{2k},
\end{align*}
which bears a striking resemblance to \eqref{fjb}. 
\end{proof}

\begin{cor} \label{rjcf_cor_2}
If $\lambda \in \C$, $\Im{(\lambda)} \neq 0$ and
\begin{align}
F_j (C_k (\lambda)) := S_k P_k \begin{bmatrix} f_{j_1} (\jordan{k}{\lambda}) & 0 \\ 0 & f_{j_2} (\jordan{k}{\bar{\lambda}}) \end{bmatrix} P_k^T \inv{S_k} \in \mc{2k}, \label{bigF}
\end{align}
where $j = \begin{pmatrix} j_1, j_2 \end{pmatrix}$ and $j_1$, $j_2 \in \{ 0, 1, \dots, p-1 \}$, 
then 
\begin{align*}
F_j (C_k (\lambda)) = 
\begin{bmatrix} 
C(f_{j_1} (\lambda)) & C(f'_{j_1} (\lambda)) & \dots   & C \left( \frac{f^{(k-1)}_{j_1} (\lambda)}{(k-1)!}  \right)	\\
                                  & C(f_{j_1} (\lambda))  & \ddots & \vdots 								\\
                      	  &                        	        & \ddots & C(f'_{j_1} (\lambda))						\\
                      	  &                        	        &            & C(f_{j_1} (\lambda))
\end{bmatrix} \in \mr{2k}
\end{align*}
if and only if $j_1 = j_2 = 0$ or $j_1 = j_2 - p$.
\end{cor}

\begin{proof}[\sc Proof]
Follows from \hyperref[lem_conj]{Lemma \ref*{lem_conj}} and \hyperref[rjcf_cor]{Corollary \ref*{rjcf_cor}}.
\end{proof}

\begin{cor} \label{rjcf_cor_3}
If $\lambda = r \exp{(i \pi)} \in \C$, where $r > 0$, and $F_j$ is defined as in \eqref{bigF}, then 
\begin{align*}
F_j (C_k (\lambda)) = 
\begin{bmatrix} 
C(f_{j_1} (\lambda)) & C(f'_{j_1} (\lambda)) & \dots   & C \left( \frac{f^{(k-1)}_{j_1} (\lambda)}{(k-1)!}  \right)	\\
                                  & C(f_{j_1} (\lambda))  & \ddots & \vdots 								\\
                      	  &                        	        & \ddots & C(f'_{j_1} (\lambda))						\\
                      	  &                        	        &            & C(f_{j_1} (\lambda))
\end{bmatrix} \in \mr{2k}
\end{align*}
if and only if $j_1 + j_2 = p - 1$.
\end{cor}

\begin{proof}[\sc Proof]
Follows from \hyperref[lem_conj2]{Lemma \ref*{lem_conj2}} and \hyperref[rjcf_cor]{Corollary \ref*{rjcf_cor}}.
\end{proof}

The next theorem provides a necessary and sufficient condition for the existence of a real $p$th-root of a real $A$ (see \cite[Theorem 2.3]{hl2011}) and our proof utilizes the real Jordan canonical form.

\begin{thm}[Existence of real $p$th-root] \label{ex_real_root_thm}
A matrix $A \in \mr{n}$ has a real $p$th-root if and only if it satisfies the ascent sequence condition specified in \hyperref[asc_seq_thm]{\rm Theorem \ref*{asc_seq_thm}} and, if $p$ is even, $A$ has an even number of Jordan blocks of each size for every negative eigenvalue. 
\end{thm}

\begin{proof}[\sc Proof]
Case 1: $p$ is even. Following \hyperref[rjcf_thm]{Theorem \ref*{rjcf_thm}}, there exists a real, invertible matrix $R$ such that 
\begin{align*}
A = R
\begin{bmatrix} 
J_0 & & &										\\
 & J_+ & & 										\\ 
 & & J_- & 										\\
 & & & C  
\end{bmatrix} \inv{R}
\end{align*}
where $J_0$ collects the singular Jordan blocks; $J_+$ collects the Jordan blocks with positive real eigenvalues; $J_-$ collects Jordan blocks with negative real eigenvalues; and $C$ collects blocks of the form \eqref{Ck_lambda} corresponding to the complex conjugate pairs of eigenvalues of $A$. 

By hypothesis, if $\jordan{k}{\lambda}$ is a submatrix of $J_-$, it must appear an even number of times; for every such pair of blocks, it follows that 
\begin{align*}
P_k [\jordan{k}{\lambda} \oplus \jordan{k}{\lambda}] P_{k}^T = C_k (\lambda),
\end{align*}
where $P_k$ is defined as in \eqref{Perm_P_k}. Thus, there exists a permutation matrix $P$ such that 
\begin{align*}
A = 
R P^T P
\begin{bmatrix} 
J_0 & & &										\\
 & J_+ & & 										\\ 
 & & J_- & 										\\
 & & & C  
\end{bmatrix} P^T P \inv{R} =
\bar{R}
\begin{bmatrix} 
J_0 & & 										\\
 & J_p & 										\\
 & & \bar{C}  
\end{bmatrix} \inv{\bar{R}}, 
\end{align*}
where $\bar{R} = R P^T$, and $\bar{C}$ collects all the blocks of the form $\eqref{Ck_lambda}$. 

Since the ascent sequence condition holds for $A$ it also holds for $J_0$, so $J_0$ has a $p$th-root $W_0$, and $W_0$ can be taken real in view of the construction given in \cite[Section 3]{p2002}; clearly, there exists a real matrix $W_+$ such that $W_+^p = J_+$ and, following Corollaries \hyperref[rjcf_cor_2]{\ref*{rjcf_cor_2}} and \hyperref[rjcf_cor_3]{\ref*{rjcf_cor_3}}, there exists a real matrix $W_c$ such that $W_c^p = \bar{C}$. Hence, the matrix $X = \bar{R} [W_0 \oplus W_+ \oplus W_c] \inv{\bar{R}}$ is a real $p$th-root of $A$.

Conversely, if $A$ satisfies the ascent sequence condition and has an odd number of Jordan blocks corresponding to a negative eigenvalue, then the process just described can not produce a real matrix $p$th-root, as one of the Jordan blocks can not be paired, so that the root of such a block is necessarily complex.  

Case 2: $p$ is odd. Follows similarly to the first case since real roots can be taken for $J_0$, $J_+$, $J_-$, and $C$.
\end{proof}

We now present an analog of \hyperref[thm_class_rts]{Theorem \ref*{thm_class_rts}} for real matrices.

\begin{thm}[Classification of $p$th-roots of nonsingular real matrices] \label{thm_class_rts2} 
Let $F_k$ be defined as in \eqref{bigF}. If $A \in \mr{n}$ is nonsingular, then $A$ has precisely $p^s$ primary $p$th-roots, given by 
\begin{align}
X_j = 
R 
\begin{bmatrix} 
\bigoplus_{k=1}^r f_{j_k} \left( J_{n_k} ( \lambda_k ) \right)  & 0 \\ 
0 & \bigoplus_{k = r + 1}^{r + c} F_{j_k} (C_{n_k} (\lambda_k)) 
\end{bmatrix} 
\inv{R}, \label{prim_roots2}		
\end{align}
where $j = \begin{pmatrix} j_1, \dots, j_r, j_{r+1},\dots,j_{r+c} \end{pmatrix}$, $j_{k} = \left( j_{k_1}, j_{k_2} \right)$ for $k=r+1,\dots,r+c$, and $j_i = j_k$ whenever $\lambda_i = \lambda_k$. 

If $s < t$, then $A$ has additional nonprimary $p$th-roots that form parameterized families of the form
\begin{align}
X_j (U) = 
R U 
\begin{bmatrix} 
\bigoplus_{k=1}^r f_{j_k} \left( J_{n_k} ( \lambda_k ) \right)  & 0	\\ 
0 & \bigoplus_{k = r + 1}^{r + c} F_{j_k} (C_{n_k} (\lambda_k)) 
\end{bmatrix} 
\inv{U} \inv{R}, \label{nnprim_roots2} 
\end{align}
where $U$ is an arbitrary nonsingular matrix that commutes with $J_\R$, and for each $j$ there exist $i$ and $k$, depending on $j$, such that $\lambda_i = \lambda_k$ while $j_i \neq j_k$.
\end{thm}

\begin{proof}[\sc Proof]
Following \hyperref[rjcf_thm]{Theorem \ref*{rjcf_thm}}, there exists a real, invertible matrix $R$ such that 
\begin{align*}
\inv{R} A  R = J_\R =
\begin{bmatrix} 
\bigoplus_{k=1}^{r} J_{n_k} (\lambda_k) & \\ & \bigoplus_{k = r + 1}^{r + c} C_{n_k} (\lambda_k)  
\end{bmatrix};
\end{align*}
if $T = \begin{bmatrix} I & \\ & \bigoplus_{k = r + 1}^{r + c} S_{n_k} P_{n_k} \end{bmatrix}$, then, following Lemmas \hyperref[rjcf_lem]{\ref*{rjcf_lem}} and \hyperref[rjcf_lem_2]{\ref*{rjcf_lem_2}}, it follows that 
\begin{align*}
\inv{T} J_\R T = J =
\begin{bmatrix}
\bigoplus_{k=1}^r \jordan{n_k}{\lambda_k} 	& 										\\
& \bigoplus_{k = r + 1}^{r + c} \left[ \jordan{n_k}{\lambda_k} \oplus \jordan{n_k}{\bar{\lambda}_k} \right] 
\end{bmatrix}.
\end{align*}
Following \hyperref[thm_class_rts]{Theorem \ref*{thm_class_rts}}, $J_\R$ has $p^s$ primary roots given by
\begin{align*}
&T \begin{bmatrix}
\bigoplus_{k=1}^r f_{j_k} (\jordan{n_k}{\lambda_k}) 	& 									\\
& \bigoplus_{k = r + 1}^{r + c} \left[ f_{j_{k_1}} (\jordan{n_k}{\mu_k}) \oplus f_{j_{k_2}} ( \jordan{n_k}{\bar{\mu}_k} \right]	
\end{bmatrix} \inv{T}															\\
&= 
\begin{bmatrix} 
\bigoplus_{k=1}^r f_{j_k} \left( J_{n_k} ( \lambda_k ) \right)  & 0 \\ 
0 & \bigoplus_{k = r + 1}^{r + c} F_{j_k} \left( C_{n_k} (\mu_k) \right) 
\end{bmatrix},
\end{align*}
where $j_k = \left( j_{k_1}, j_{k_2} \right)$ for $k = r + 1, \dots, r + c$, which establishes \eqref{prim_roots2}.

If $A$ is derogatory, then $J_\R$ has additional roots of the form 
\begin{align*}
&T W \begin{bmatrix}
\bigoplus_{k=1}^r f_{j_k} (\jordan{n_k}{\lambda_k}) 	& 									\\
& \bigoplus_{k = r + 1}^{r + c} \left[ f_{j_{k_1}} (\jordan{n_k}{\mu_k}) \oplus f_{j_{k_2}} ( \jordan{n_k}{\bar{\mu}_k} \right]	
\end{bmatrix} \inv{W} \inv{T},
\end{align*}
where $W$ is any matrix that commutes with $J$. Note that 
\begin{align*}
&T W \begin{bmatrix}
\bigoplus_{k=1}^r f_{j_k} (\jordan{n_k}{\lambda_k}) 	& 									\\
& \bigoplus_{k = r + 1}^{r + c} \left[ f_{j_{k_1}} (\jordan{n_k}{\mu_k}) \oplus f_{j_{k_2}} ( \jordan{n_k}{\bar{\mu}_k} \right]	
\end{bmatrix} \inv{W} \inv{T}	\\
&=  U \begin{bmatrix}
\bigoplus_{k=1}^r f_{j_k} (\jordan{n_k}{\lambda_k}) 	& 									\\
& \bigoplus_{k = r + 1}^{r + c} F_{j_k} (C_{n_k} (\lambda_k))	
\end{bmatrix} \inv{U},
\end{align*}
where $U = T W \inv{T}$. Following \cite[Theorem 1, \S 12.4]{lt1985}, $U$ is an arbitary, nonsingular matrix that commutes with $J_\R$, which establishes \eqref{nnprim_roots2}.
\end{proof}

The next theorem identifies the number of real primary $p$th-roots of a real matrix (c.f. \cite[Theorem 2.4]{hl2011}) and our proof utilizes the real Jordan canonical form. 

\begin{cor} \label{real_rts_thm} 
Let the nonsingular real matrix $A$ have $r_1$ distinct positive real eigenvalues, $r_2$ distinct negative real eigenvalues, and $c$ distinct complex-conjugate pairs of eigenvalues. If $p$ is even, there are (a) $2^{r_1} p^c$ real primary $p$th-roots when $r_2 = 0$; and (b) no real primary $p$th-roots when $r_2 >0$. If $p$ is odd, there are $p^c$ real primary $p$th-roots.
\end{cor}

\begin{proof}[\sc Proof]
Following \hyperref[rjcf_thm]{Theorem \ref*{rjcf_thm}}, there exists a real, invertible matrix $R$ such that 
\begin{align*}
\inv{R} A  R = J_\R =
\begin{bmatrix} 
\bigoplus_{k=1}^{r} J_{n_k} (\lambda_k) & \\ & \bigoplus_{k = r + 1}^{r + c} C_{n_k} (\lambda_k)  
\end{bmatrix}.
\end{align*}
Case 1: $p$ is even. If $r_2 > 0$, then $A$ does not possess a real primary root, since $A$ must have an even number of Jordan blocks of each size for every negative eigenvalue, and the same branch of the $p\tth$-root function must be selected for every Jordan block containing the same negative eigenvalue. If $r_2 = 0$, then, following \hyperref[rjcf_cor_2]{Corollary \ref*{rjcf_cor_2}}, for every complex-conjugate pair of eigenvalues, there are $p$ choices such that $F_{j_k} \left( C_{n_k} (\lambda_k) \right)$ is real. For every real eigenvalue, there are two choices such that $f_{j_k} \left( J_{n_k} ( \lambda_k ) \right)$ is real, yielding $2^{r_1} p^c$ real primary roots. 

Case 2: $p$ is odd. The matrix $X_j$ is real provided that the principal-branch of the $p\tth$-root function is chosen for every real eigenvalue. Similar to the first case, there are $p$ choices such that $F_{j_k} \left( C_{n_k} (\lambda_k) \right)$ is real, yielding $p^c$ real primary roots.
\end{proof}

The following theorem extends \hyperref[thm_class_rts]{Theorem \ref*{thm_class_rts}} to include singular matrices (see \cite[Theorem 2.6]{hl2011}).

\begin{thm}[Classification of $p$th-roots] \label{nnsing_thm}
Let $A \in \mc{n}$ have the Jordan canonical form $\inv{Z} A Z = J = J_0 \oplus J_1$, where $J_0$ collects together all the Jordan blocks corresponding to the eigenvalue zero and $J_1$ contains the remaining Jordan blocks. If $A$ possesses a $p$th-root, then all $p$th-roots of $A$ are given by $A = Z \left( X_0 \oplus X_1 \right) \inv{Z}$, where $X_1$ is any $p$th-root of $J_1$, characterized by \hyperref[thm_class_rts]{\rm Theorem \ref*{thm_class_rts}}, and $X_0$ is any $p$th-root of $J_0$.
\end{thm}

\section{Main Results} \label{main_results}

We now focus on the $p$th-roots of primitive matrices and matrices possessing the strong Perron-Frobenius property.

Recall that a matrix $A \in \mr{n}$ is {\it reducible} if $n \geq 2$ and there is a permutation matrix $P$ such that 
\[ A = P^T \begin{bmatrix} A_{11} & A_{12} \\ 0 & A_{22} \end{bmatrix} P \]
where $A_{11}$ and $A_{22}$ are square, nonempty submatrices. A matrix $A$ is {\it irreducible} if it is not reducible. A matrix $A = [a_{ij}] \in \mr{n}$ is said to be (entrywise) {\it nonnegative} (respectively, {\it positive}), denoted $A \geq 0$ (respectively, $A >0$), if $a_{ij} \geq 0$ (respectively, $a_{ij} > 0$) for all $1 \leq i,j \leq n$. Recall from the \hyperref[intro]{Introduction} that a matrix $A \in \mr{n}$ is {\it eventually positive} ({\it nonnegative}) if there exists a nonnegative integer $p$ such that $A^k$ is entrywise positive (nonnegative) for all $k \geq p$. If $p$ is the smallest such integer, then $p$ is called the {\it power index of $A$} and is denoted by $p(A)$.

We recall the Perron-Frobenius theorem for positive matrices (see \cite[Theorem 8.2.11]{hj1990}).

\begin{thm} \label{pf_thm} 
If $A \in \mr{n}$ is positive, then 
\begin{enumerate}
\item[(a)] $\rho := \sr{A} > 0$;
\item[(b)] $\rho \in \sig{A}$;
\item[(c)] there exists a positive vector $x$ such that $Ax = \rho x$;
\item[(d)] $\rho$ is a simple eigenvalue of $A$.
\item[(e)] $|\lambda| < \rho$ for every $\lambda \in \sig{A}$ such that $\lambda \neq \rho$. 
\end{enumerate}
\end{thm}

There are nonnegative matrices containing entries that are zero that satisfy \hyperref[pf_thm]{Theorem \ref*{pf_thm}}. Recall that a nonnegative matrix $A \in \mr{n}$ is said to be {\it primitive} if it is irreducible and has only one eigenvalue of maximum modulus. The conclusions to \hyperref[pf_thm]{Theorem \ref*{pf_thm}} apply to primitive matrices (see \cite[ Theorem 8.5.1]{hj1990}), and the following theorem is a useful characterization of primitivity (see \cite[Theorem 8.5.2]{hj1990}).

\begin{thm} 
If $A \in \mr{n}$ is nonnegative, then $A$ is primitive if and only if $A^k > 0$ for some $k \geq 1$. 
\end{thm}

One can verify that the matrix $\begin{bmatrix} 2 & 1 \\ 2 & -1 \end{bmatrix}$ possesses properties (a) through (e) of \hyperref[pf_thm]{Theorem \ref*{pf_thm}}, is irreducible, but obviously contains a negative entry. This motivates the following concept.

\begin{mydef} 
A matrix $A \in \mr{n}$ is said to possess the {\it strong Perron-Frobenius property} if $A$ possesses properties (a) through (e) of \hyperref[pf_thm]{Theorem \ref*{pf_thm}}.
\end{mydef}

The following theorem characterizes the strong Perron-Frobenius property (see \cite[Lemma 2.1]{h1981}, \cite[Theorem 1]{jt2004}, or \cite[Theorem 2.2]{n2006}).

\begin{thm} \label{evpos_thm}
A real matrix $A$ is eventually positive if and only if $A$ and $A^T$ possess the strong Perron-Frobenius property. 
\end{thm}

We now present our main results.

\begin{thm} \label{main_thm} 
Let the nonsingular primitive matrix $A$ have $r_1$ distinct positive real eigenvalues, $r_2$ distinct negative real eigenvalues, and $c$ distinct complex-conjugate pairs of eigenvalues. If $p$ is even, there are (a) $2^{r_1-1} p^c$ eventually positive primary $p$th-roots when $r_2 = 0$; and (b) no eventually positive primary $p$th-roots if $r_2 >0$. If $p$ is odd, there are $p^c$ eventually positive primary $p$th-roots.
\end{thm}

\begin{proof}[\sc Proof] 
If $r_2 > 0$, then, following \hyperref[real_rts_thm]{Theorem \ref*{real_rts_thm}}, the matrix $A$ does not have a real primary $p$th-root, hence, {\it a fortiori}, it can not have an eventually positive primary $p$th-root. 

If $r_2 = 0$, then, following Theorems \hyperref[rjcf_thm]{\ref*{rjcf_thm}} and \hyperref[pf_thm]{\ref*{pf_thm}}, there exists a real, invertible matrix $R$ such that   
\begin{align}
A = \begin{bmatrix} x & R' \end{bmatrix} 
\begin{bmatrix} 
\rho &                                                                         &  								\\
        & \bigoplus_{k=2}^r \jordan{n_k}{\lambda_k} & 								\\
        & 						         & \bigoplus_{k=r+1}^{r+c} C_{n_k} (\lambda_k) 
\end{bmatrix}
\begin{bmatrix} y^T \\ (\inv{R})' \end{bmatrix},
\end{align} 
where $R = \begin{bmatrix} x & R' \end{bmatrix} \in \mr{n}$, $\inv{R} = \begin{bmatrix} y \\ (\inv{R})' \end{bmatrix} \in \mr{n}$, $x > 0$ is the right Perron-vector, and $y>0$ is the left Perron-vector. 

Because $| f_j (z) | < |f_{j'} (z')|$ for all $j$, $j' \in \{ 0, 1, \dots, p-1 \}$, any primary $p\tth$-root $X$ of the form
\begin{align*}
X_j = R 
\begin{bmatrix} 
\sqrt[p]{\rho} & & 						\\
 & \bigoplus_{k=2}^r f_{j_k} (\jordan{n_k}{\lambda_k}) & 	\\
 & & \bigoplus_{k=r+1}^{r+c} F_{j_k} (C_{n_k} (\lambda_k)) 
\end{bmatrix}
\inv{R}
\end{align*}
inherits the strong Perron-Frobenius property from $A$. Since $f(A)^T = f(A^T)$ (\cite[Theorem 1.13(b)]{h2008}), a similar argument demonstrates that $X^T$ inherits the strong Perron-Frobenius property from $A^T$. 
 
Case 1: $p$ is even. For all $k = 2,\dots,r$, there are two possible choices such that $f_{j_k} \left( J_{n_k} ( \lambda_k ) \right)$ is real;  following \hyperref[rjcf_cor_2]{Corollary \ref*{rjcf_cor_2}}, for all $k = r+1, \dots, r+c$, there are $p$ choices such that $F_{j_k} \left( C_{n_k} (\lambda_k) \right)$ is real. Thus, there are $2^{r_1-1}p^c$ possible ways to select $X$ and $X^T$ to be real.

Case 2: $p$ is odd. For all $k = 2,\dots,r$, the principal-branch of the $p\tth$-root function must be selected so that $f_{j_k} \left( J_{n_k} ( \lambda_k ) \right)$ is real and, similar to the previous case, there are $p$ choices such that $F_{j_k} \left( C_{n_k} (\lambda_k) \right)$ is real. Hence, there are $p^c$ ways to select $X$ and $X^T$ real.

In either case, following \hyperref[evpos_thm]{Theorem \ref*{evpos_thm}}, the matrices $X$ and $X^T$ are eventually positive.
\end{proof}

\begin{ex} \label{thm_example}
We demonstrate \hyperref[main_thm]{Theorem \ref*{main_thm}} via an example. Consider the matrix 
\begin{align*}
A = \frac{1}{5}
\begin{bmatrix}
16 & 16 & 6 & 11 & 1	\\
 7 & 12 & 12 & 12 & 7 	\\
 9 & 4 & 14 & 9 & 14 	\\
 8 & 8 & 8 & 13 & 13 	\\
10 & 10 & 10 & 5 & 15 	
\end{bmatrix}.
\end{align*}
Note that $A = R J_\R \inv{R}$, where 
\begin{align*}
J_\R = \kbordermatrix{
& & & & & 				\\
& 10 & \vrule & 0 & 0 & 0 & 0 	\\
\cline{2-7}
& 0 & \vrule & 1 & 1 & 1 & 0 	\\
& 0 & \vrule & -1 & 1 & 0 & 1 	\\
& 0 & \vrule & 0 & 0 & 1 & 1 	\\
& 0 & \vrule & 0 & 0 & -1 & 1 	
}, \hspace{3pt} R = 
\begin{bmatrix} 
1 & 1 & 1 & 1 & 1 	\\ 
1 & -1 & 0 & 0 & 0 	\\  
1 & 0 & -1 & 0 & 0 	\\ 
1 & 0 & 0 & -1 & 0  	\\ 
1 & 0 & 0 & 0 & -1  
\end{bmatrix},
\end{align*}  
and 
\begin{align*}
\inv{R} = \frac{1}{5}
\begin{bmatrix} 
1 & 1 & 1 & 1 & 1 	\\ 
1 & -4 & 1 & 1 & 1 	\\  
1 & 1 & -4 & 1 & 1 	\\ 
1 & 1 & 1 & -4 & 1  	\\ 
1 & 1 & 1 & 1 & -4  
\end{bmatrix}.
\end{align*} 
Because $\sig{A} = \{ 10, 1 + \ii, 1 + \ii, 1 - \ii, 1 - \ii  \}$, following \hyperref[main_thm]{Theorem \ref*{main_thm}}, $A$ has eight primary matrix square-roots, of which the matrices
\vspace{-12pt}
\begin{align*} 
X_j &= 
R \kbordermatrix{ 							\\
& \sqrt{10} & \vrule & 0 & 0 & 0 & 0				\\
\cline{2-7}
 & 0 & \vrule & 1.0987 & 0.4551 & 0.3884 &  -0.1609		\\
 & 0 & \vrule & -0.4551 & 1.0987  &  0.1609 &  0.3884	\\
 & 0 & \vrule & 0 & 0 & 1.0987 & 0.4551			\\
 & 0 & \vrule & 0 & 0 &  -0.4551 & 1.0987} \inv{R}		\\
&= \begin{bmatrix}
1.6668 & 1.0232  & 0.1130 & 0.4738 & -0.1145	\\
0.2762 & 1.3749  & 0.7313 & 0.6646 & 0.1153	\\
0.3939 & -0.0612 & 1.4926 & 0.5548 & 0.7823	\\
0.3217 & 0.3217  & 0.3217 & 1.4204 & 0.7768	\\
0.5037 & 0.5037 & 0.5037 & 0.0486 & 1.6024
\end{bmatrix},
\end{align*} 
where $j = (0,(0,0))$, and
\vspace{-12pt}
\begin{align*}
X_{j'} &= R
\kbordermatrix{ 							\\
& \sqrt{10} & \vrule & 0 & 0 & 0 & 0				\\
\cline{2-7}
 & 0 & \vrule & -1.0987 & -0.4551 & -0.3884 &  0.1609	\\
 & 0 & \vrule & 0.4551 & -1.0987  &  -0.1609 &  -0.3884	\\
 & 0 & \vrule & 0 & 0 & -1.0987 & -0.4551			\\
 & 0 & \vrule & 0 & 0 &  0.4551 & -1.0987} \inv{R}		\\
&= \begin{bmatrix}
-0.4019 & 0.2417 & 1.1519 & 0.7911 & 1.3794 \\
0.9887 & -0.1100 & 0.5336 & 0.6003 & 1.1496 \\
0.8710 & 1.3261 & -0.2276 & 0.7101 & 0.4826 \\
0.9432 & 0.9432 & 0.9432 & -0.1555 & 0.4881 \\
0.7612 & 0.7612 & 0.7612 & 1.2163 & -0.3375
\end{bmatrix},
\end{align*}  where $j' = (0,(1,1))$, are eventually positive square-roots of $A$.
\end{ex}

The following question arises from \hyperref[thm_example]{Example \ref*{thm_example}}: are $X_j$ and $X_{j'}$ the only eventually positive square-roots of $A$? This is answered in the following result, which yields an explicit description of the eventually positive primary roots of a nonsingular primitive matrix $A$.

\begin{thm} \label{main_thm_2}
Let $A$ be a nonsingular, primitive matrix and $X_j$ be any primary $p$th-root of $A$ of the form  
\begin{align*}
X_j = R
\begin{bmatrix} 
f_{j_1} (\rho) & & 							\\
 & \bigoplus_{k=2}^r f_{j_k} (\jordan{n_k}{\lambda_k}) & 	\\
 & & \bigoplus_{k=r+1}^{r+c} F_{j_k} (C_{n_k} (\lambda_k)) 
\end{bmatrix} \inv{R},
\end{align*}
where  $j =  \left( j_1, \dots,j_r, j_{r+1}, \dots, j_{r+c} \right)$ and $j_k = ( j_{k_1}, j_{k_2})$, for $k = r+1,\dots,r+c$. If $p$ is odd, then $X_j$ is eventually positive if and only if 
\begin{enumerate}
\item $j_1 = 0$; 
\item $j_k = 0$ for all $k = 2,\dots,r$;  and
\item $j_k = (0,0)$ or $j_k = ( j_{k_1},p -  j_{k_1} )$ for all $k = r+1, \dots, r+c$.
\end{enumerate}
If $p$ is even, then 
$X_j$ is eventually positive if and only if 
\begin{enumerate}
\item $j_1 = 0$; 
\item $j_k = 0$ or $j_k = p/2$ for all $k = 2,\dots,r$;  and
\item $j_k = (0,0)$ or $j_k = ( j_{k_1},p -  j_{k_1} )$ for all $k = r+1, \dots, r+c$.
\end{enumerate}
\end{thm}

\begin{proof}[\sc Proof] 
We demonstrate {\it necessity} as {\it sufficiency} is shown in the proof of \hyperref[main_thm]{Theorem \ref*{main_thm}}. 

To this end, we demonstrate the contrapositive. Case 1: $p$ is odd. If the principal branch of the $p\tth$-root function is not selected for the Perron eigenvalue, then $X_j$ can not possess the strong Perron-Frobenius property; if $j_k \neq 0$ for some  $k \in \{ 2,\dots,r \}$, then $f_{j_k}(\jordan{n_k}{\lambda_k})$ is not real so that $X_j$ can not be real; similarly, $X_j$ is not real if  $j_k \neq (0,0)$ or $j_k \neq (j_{k_1},p -  j_{k_1})$ for some $k \in \{ r+1, \dots, r+c \}$.

Case 2: $p$ is even. Result is similar to the first case, but we note that, without loss of generality, we may assume that $A$ does not have any negative eigenvalues (else it can not possess a primary root).  
\end{proof}

\begin{thm} \label{main_thm_3}
Let $A$ be a primitive, nonsingular, derogatory matrix that possesses a real root, and let $X_j (U)$ be any nonprimary root of $A$ of the form
\begin{align*}
X_j (U) = 
R U 
\begin{bmatrix}
f_{j_1} (\rho) & &          							\\
& \bigoplus_{k=2}^r f_{j_k} \left( J_{n_k} ( \lambda_k ) \right)  & \\ 
& & \bigoplus_{k = r + 1}^{r + c} F_{j_k} (C_{n_k} (\lambda_k)) 
\end{bmatrix} 
\inv{U} \inv{R},
\end{align*}
where  $j =  \left( j_1, \dots,j_r, j_{r+1}, \dots, j_{r+c} \right)$ and $j_k = ( j_{k_1}, j_{k_2})$, for $k = r+1,\dots,r+c$. If $p$ is even, then $X_j (U)$ is eventually positive if and only if
\begin{enumerate}
\item $j_1 = 0$;
\item $j_k = 0$, or $j_k = p/2$, for all $k = 2,\dots,r$; 
\item $j_k = (0,0)$ or $j_k = ( j_{k_1},p -  j_{k_1} )$ for all $k = r+1, \dots, r+c$;
\item $U$ is selected to be real and nonsingular; and
\item Jordan blocks containing negative eigenvalues are transformed, via a permutation matrix, to blocks of the form \eqref{Ck_lambda} (see proof of \hyperref[ex_real_root_thm]{\rm Theorem \ref*{ex_real_root_thm}}) and branches for these blocks are selected in accordance with \hyperref[rjcf_cor_3]{\rm Corollary \ref*{rjcf_cor_3}}
\end{enumerate}
subject to the constraint that for each $j$, there exist $i$ and $k$, depending on $j$, such that $\lambda_i = \lambda_k$, while $j_i \neq j_k$.

\noindent If $p$ is odd, then $X_j (U)$ is eventually positive if and only if
\begin{enumerate}
\item $j_1 = 0$;
\item $j_k = 0$ for all $k = 2,\dots,r$; 
\item $j_k = (0,0)$ or $j_k = ( j_{k_1},p -  j_{k_1} )$ for all $k = r+1, \dots, r+c$; and
\item $U$ is selected to be real and nonsingular;
\end{enumerate}
subject to the constraint that for each $j$, there exist $i$ and $k$, depending on $j$, such that $\lambda_i = \lambda_k$, while $j_i \neq j_k$.
\end{thm}

\begin{ex} \label{thm_example2}
We demonstrate \hyperref[main_thm_3]{Theorem \ref*{main_thm_3}} via an example. Consider the matrix 
\begin{align*}
A = \frac{1}{9}
\begin{bmatrix}
32 & 32 & 14 & 23 & 5 & 32 & 14	& 23 & 5	\\
17 & 26 & 26 & 26 & 17 & 17 & 17 & 17 & 17 	\\
19 & 10 & 28 & 19 & 28 & 19 & 19 & 19 & 19	\\
18 & 18 & 18 & 27 & 27 & 18 & 18 & 18 & 18	\\
20 & 20 & 20 & 11 & 29 & 20 & 20 & 20 & 20	\\
17 & 17 & 17 & 17 & 17 & 26 & 26 & 26 & 17	\\
19 & 19 & 19 & 19 & 19 & 10 & 28 & 19 & 28	\\
18 & 18 & 18 & 18 & 18 & 18 & 18 & 27 & 27	\\
20 & 20 & 20 & 20 & 20 & 20 & 20 & 11 & 29
\end{bmatrix}.
\end{align*}
Note that $A = R J_\R \inv{R}$, where
\vspace{-12pt}
\begin{align*}
J_\R = \kbordermatrix{ 						\\
& 20 & \vrule & 0 & 0 & 0 & 0 & \vrule & 0 & 0 & 0 & 0 	\\
\cline{2-12}
& 0 & \vrule & 1 & 1 & 1 & 0 & \vrule & 0 & 0 & 0 & 0	\\
& 0 & \vrule & -1 & 1 & 0 & 1 & \vrule & 0 & 0 & 0 & 0	\\
& 0 & \vrule & 0 & 0 & 1 & 1 & \vrule & 0 & 0 & 0 & 0	\\
& 0 & \vrule & 0 & 0 & -1 & 1 & \vrule & 0 & 0 & 0 & 0	\\
\cline{2-12}
& 0 & \vrule & 0 & 0 & 0 & 0 & \vrule & 1 & 1 & 1 & 0	\\
& 0 & \vrule & 0 & 0 & 0 & 0 & \vrule & -1 & 1 & 0 & 1	\\
& 0 & \vrule & 0 & 0 & 0 & 0 & \vrule & 0 & 0 & 1 & 1 	\\
& 0 & \vrule & 0 & 0 & 0 & 0 & \vrule & 0 & 0 & -1 & 1	 	
}
\end{align*}
\begin{align*}
R = 
\begin{bmatrix} 
1 & 1 & 1 & 1 & 1 & 1 & 1 & 1 & 1 	\\ 
1 & -1 & 0 & 0 & 0 & 0 & 0 & 0 & 0  	\\  
1 & 0 & -1 & 0 & 0 & 0 & 0 & 0 & 0	\\ 
1 & 0 & 0 & -1 & 0 & 0 & 0 & 0 & 0	\\ 
1 & 0 & 0 & 0 & -1 & 0 & 0 & 0 & 0 	\\ 
1 & 0 & 0 & 0 & 0 & -1 & 0 & 0 & 0 	\\
1 & 0 & 0 & 0 & 0 & 0 & -1 & 0 & 0 	\\
1 & 0 & 0 & 0 & 0 & 0 & 0 & -1 & 0 	\\
1 & 0 & 0 & 0 & 0 & 0 & 0 & 0 & -1 	
\end{bmatrix},
\end{align*}  
and 
\begin{align*}
\inv{R} = \frac{1}{9}
\begin{bmatrix} 
1 & 1 & 1 & 1 & 1 & 1 & 1 & 1 & 1 	\\ 
1 & -8 & 1 & 1 & 1 & 1 & 1 & 1 & 1 	\\  
1 & 1 & -8 & 1 & 1 & 1 & 1 & 1 & 1	\\ 
1 & 1 & 1 & -8 & 1 & 1 & 1 & 1 & 1 	\\ 
1 & 1 & 1 & 1 & -8 & 1 & 1 & 1 & 1 	\\
1 & 1 & 1 & 1 & 1 & -8 & 1 & 1 & 1 	\\
1 & 1 & 1 & 1 & 1 & 1 & -8 & 1 & 1 	\\
1 & 1 & 1 & 1 & 1 & 1 & 1 & -8 & 1 	\\
1 & 1 & 1 & 1 & 1 & 1 & 1 & 1 & -8 	
\end{bmatrix}.
\end{align*} 
If $j = (0, (0,0),(1,1))$, then, following \hyperref[main_thm_3]{Theorem \ref*{main_thm_3}}, any matrix of the form 
\begin{align*} 
X_j (U) = 
R U
\begin{bmatrix}
\sqrt{20} &						\\
 & F_{(0,0)} (C_2(1+\ii))	&			\\
 & & F_{(1,1)} (C_2(1+\ii))
\end{bmatrix}
\inv{R} \inv{U},
\end{align*} 
where 
\begin{align*}
U = 
\kbordermatrix{ 							\\
& u_1 & \vrule & 0 & 0 & 0 & 0 & \vrule & 0 & 0 & 0 & 0 	\\ 
\cline{2-12}
& 0 & \vrule & u_2 & u_3 & 1 & 0 & \vrule & u_4 & u_5 & 1 & 0 	\\  
& 0 & \vrule & -u_3 & u_2 & 0 & 1 & \vrule & -u_5 & u_4 & 0 & 1	\\ 
& 0 & \vrule & 0 & 0 & u_2 & u_3 & \vrule & 0 & 0 & u_4 & u_5 	\\ 
& 0 & \vrule & 0 & 0 & -u_3 & u_2 & \vrule & 0 & 0 & -u_5 & u_4 	\\
\cline{2-12}
& 0 & \vrule & u_6 & u_7 & 1 & 0 & \vrule & u_8 & u_9 & 1 & 0 	\\
& 0 & \vrule & -u_7 & u_6 & 0 & 1 & \vrule & -u_9 & u_8 & 0 & 1 	\\
& 0 & \vrule & 0 & 0 & u_6 & u_7 & \vrule & 0 & 0 & u_8 & u_9 	\\
& 0 & \vrule & 0 & 0 & -u_7 & u_6 & \vrule & 0 & 0 & -u_9 & u_8} \in \mr{9},
\end{align*} 
$\det{(U)} \neq 0$, is an eventually positive square-root of $A$.
\end{ex}

Next, we present an analog of \hyperref[nnsing_thm]{Theorem \ref*{nnsing_thm}} for real matrices.

\begin{thm} \label{sing_evproots_thm}
Let the primitive matrix $A \in \mr{n}$ have the real Jordan canonical form $\inv{R} A R = J_\R = J_0 \oplus J_1$, where $J_0$ collects all the singular Jordan blocks and $J_1$ collects the remaining Jordan blocks. If $A$ possesses a real root, then all eventually positive $p$th-roots of $A$ are given by $A = R \left( X_0 \oplus X_1 \right) \inv{R}$, where $X_1$ is any $p$th-root of $J_1$, characterized by \hyperref[main_thm_2]{\rm Theorem \ref*{main_thm_2}} or \hyperref[main_thm_3]{\rm Theorem \ref*{main_thm_3}}, and $X_0$ is a real $p$th-root of $J_0$.
\end{thm}

\begin{rem} 
It should be clear that Theorems \hyperref[main_thm]{\ref*{main_thm}}, \hyperref[main_thm_2]{\ref*{main_thm_2}}, \hyperref[main_thm_3]{\ref*{main_thm_3}}, and \hyperref[sing_evproots_thm]{\ref*{sing_evproots_thm}} remain true if the assumption of primitivity  is replaced with eventually positivity.
\end{rem}

Recall that for $A \in \mc{n}$ with no eigenvalues on $\R^-$, the {\it principal} $p\tth$-root, denoted by $A^{1/p}$, is the unique $p\tth$-root of $A$ all of whose eigenvalues lie in the segment $\{ z : -\pi/p < \arg(z) < \pi/p\}$ \cite[Theorem 7.2]{h2008}. In addition, recall that a nonnegative matrix $A$ is said to be {\it stochastic} if $\sum_{j=1}^n a_{ij} = 1$, for all $i=1,\dots,n$. 

In \cite{hl2011}, being motivated by discrete-time Markov-chain applications, two classes of stochastic matrices were identified that possess stochastic principal $p\tth$-roots for all $p$. As a consequence, and of particular interest, for these classes of matrices the twelfth-root of an annual transition matrix is itself a transition matrix. A (monthly) transition matrix that contains a negative entry or is complex is meaningless in the context of a model, however the following remark demonstrates that, under suitable conditions, a matrix-root of a primitive stochastic matrix will be \emph{eventually stochastic} (i.e., eventually positive with row sums equal to one).

Eventual stochasticity may be useful in the following manner: consider, for example, the application of discrete-time Markov chains in credit risk: let $P=[p_{ij}] \in \mr{n}$ be a primitive transition matrix where $p_{ij} \in [0,1]$ is the probability that a firm with rating $i$ transitions to rating $j$. Such matrices are derived via annual data, and the twelfth-root of such a matrix would correspond to a monthly transition matrix if the entries are nonnegative. 

In application, the twelfth-root would be used for \emph{forecasting} purposes (e.g., to estimate the likelihood that a firm with credit-rating $i$, transitions to rating $j$, $m$ months into the future). Thus, if $R$ is the twelfth-root of $P$ and $R^m \geq 0$, then $r_{ij}^{(m)}$ is a candidate for the aforementioned probability. Moreover, it is well-known that $n^2 - 2n +2$ is a sharp upper-bound for the \emph{primitivity index} of a primitive matrix (see, e.g., \cite[Corollary 8.5.9]{hj1990}) so that eventual stochasticity is feasible in application.  

\begin{rem} 
Theorems \hyperref[main_thm]{\ref*{main_thm}}, \hyperref[main_thm_2]{\ref*{main_thm_2}}, \hyperref[main_thm_3]{\ref*{main_thm_3}}, and \hyperref[sing_evproots_thm]{\ref*{sing_evproots_thm}} remain true if stochasticity is added as an assumption on the matrix $A$ and the conclusion of eventually postivity is replaced with eventually stochasticity.
\end{rem}

We conclude with the following remark.

\begin{rem}
If $A$ is an eventually positive matrix matrix, then $B := A^q$ is eventually positive for all $q \in \N$: thus, all the previous results hold for rational powers of $A$.  
\end{rem}

\bibliographystyle{model1-num-names}
\bibliography{laabib}

\end{document}